\documentclass{amsart}
\usepackage{graphicx}
\usepackage{amsfonts,amsmath,amstext,amsthm,amssymb}

\numberwithin{equation}{section}

\newtheorem{theorem}{Theorem}[section]
\newtheorem{lemma}[theorem]{Lemma}
\newtheorem{corollary}[theorem]{Corollary}

\theoremstyle{definition}

\newtheorem{OldTheorem}{Theorem}

\theoremstyle{remark}

\def\ZI{\ensuremath{\mathbb I}}
\def\ZN{\ensuremath{\mathbb N}}

\def\zP{\ensuremath{\mathcal P}}
\def\ZT{\ensuremath{\mathbb T}}
\def\ZR{\ensuremath{\mathbb R}}
\def\md#1#2\emd{\ifx0#1
\begin{equation*} #2 \end{equation*}\fi  %  single line display, no number
\ifx1#1\begin{equation}#2\end{equation}\fi   % single line display, number
\ifx2#1\begin{align*}#2\end{align*}\fi   % aligned display, no number
\ifx3#1\begin{align}#2\end{align}\fi    % aligned display, number
\ifx4#1\begin{gather*}#2\end{gather*}\fi  % multline, not align, no number
\ifx5#1\begin{gather}#2\end{gather}\fi   % multinline, not align
\ifx6#1\begin{multline*}#2\end{multline*}\fi  %  display too long for one line
\ifx7#1\begin{multline}#2\end{multline}\fi  % as above, with numbers
\ifx8#1\begin{multline*}\begin{split}#2\end{split}\end{multline*}\fi
\ifx9#1\begin{multline}\begin{split}#2\end{split}\end{multline}\fi
}
\newcommand {\e }[1]{(\ref{#1})}
\newcommand {\lem }[1]{Lemma \ref{#1}}
\newcommand {\trm }[1]{Theorem \ref{#1}}
\newcommand {\otrm }[1]{Theorem \ref{#1}}

\begin{document}

\title{On a theorem of Littlewood}%
\author{G. A. Karagulyan }

\address{Institute of Mathematics of Armenian
National Academy of Sciences, Baghramian Ave.- 24b, 0019,
Yerevan, Armenia}%
\email{g.karagulyan@yahoo.com}%
\author{ M. H. Safaryan}
\address{Yerevan State University,
Alek Manukyan, 1, 0049,
Yerevan, Armenia}%
\email{mher.safaryan@gmail.com }

%\date{}%
%\dedicatory{}%
%\commby{}%
% ----------------------------------------------------------------
\begin{abstract}
In 1927 Littlewood constructed an example of bounded holomorphic function on the unit disk, which diverges almost everywhere along rotated copies of any given curve in the unit disk ending tangentially to the boundary. This theorem was the complement of a positive theorem of Fatou 1906, establishing almost everywhere nontangential convergence of bounded holomorphic functions.  There are several generalizations of the Littlewood's  theorem which proofs are based on the specific properties of Poisson kernel. We generalize Littlewood's theorem for operators having general kernels.
\end{abstract}
\maketitle
\section{Introduction}
A classical theorem of Fatou \cite{Fat} asserts that any bounded analytic function on the unit disc $D=\{|z|<1\}$ has nontangential limit for almost all boundary points. This theorem has been generalized in different directions. It is well known that the functions from Hardy spaces and  harmonic functions posses the a.e. nontangential convergence property too. J. Littlewood \cite{Lit} made an important complement to these results, proving essentiality of nontangential approach in Fatou's theorem. The following formulation of Littlewood's theorem fits to the further aim of the present paper.
\begin{OldTheorem}[Littlewood, 1927]
If a continuous function $\lambda(r):[0,1]\to \ZR$ satisfies the conditions
 \md1\label{0-1}
\lambda(1)=0,\quad  \lim_{r\to 1}\lambda(r)/(1-r)=\infty,
 \emd
 then there exists a bounded analytic function $f(z)$, $z\in D$, such that the boundary limit
 \md0
 \lim_{r\to 1}f\left(re^{i\left(x+\lambda(r)\right)}\right)
 \emd
 does not exist almost everywhere on $\ZT$.
\end{OldTheorem}
 A simple proof of this theorem was given by Zygmund \cite{Zyg}.
In \cite{LoPi} Lohwater and Piranian proved, that in Littlewood's theorem almost everywhere divergence can be replaced to everywhere. Moreover, it is proved
 \begin{OldTheorem}[Lohwater and Piranian, 1957]\label{OT1}
 If a continuous function $\lambda(r)$  satisfies \e{0-1}, then there exists a Blaschke product $B(z)$ such that the limit
 \md0
 \lim_{r\to 1}B\left(re^{i\left(x+\lambda(r)\right)}\right)
 \emd
 does not exist for any $x\in \ZT$.
 \end{OldTheorem}

 In \cite{Aik1} Aikawa  obtained a similar everywhere divergence theorem for bounded harmonic functions on the unit disk, giving a positive answer to a problem raised by Barth [\cite{Bar}, p. 551].
 \begin{OldTheorem}[Aikawa]
 If $\lambda(r)$ is continuous and satisfies the condition \e{0-1},
 then there exists a bounded harmonic function $u(z)$ on the unit disc, such that the limit
 \md0
\lim_{r\to 1} u\left(re^{i\left(x+\lambda(r)\right)}\right)
 \emd
 does not exist for any $x\in \ZT$.
 \end{OldTheorem}

 Almost everywhere convergence over some "semi" tangential regions were investigated by Nagel and Stein \cite{NaSt}, Di Biase \cite{DiBi}, Di Biase-Stokolos-Svensson-Weiss \cite{BSSW}.
P.~Sj\"{o}gren (\cite{Sog1}, \cite{Sog2}, \cite{Sog3}), J.-O.~R\"{o}nning (\cite{Ron1}, \cite{Ron2}, \cite{Ron3}),  I.~N.~Katkovskaya and V.~G.~ Krotov (\cite{Kro})
obtained some tangential convergence properties for the square root Poisson integrals.

It is well known that the theorems A,B and C may be also formulated in the terms of Poisson integral
\md0
\zP_r(x,f)=\frac{1}{2\pi }\int_\ZT P_r(x-t) f(t)dt,
\emd
since any bounded analytic or harmonic function on the unit disc can be written in this form, where  $f$ is either in $H^\infty$ or $L^\infty$. In addition, the proofs of these theorems are based on some properties of such functions.

In the present paper we consider general operators
\md1\label{Phi}
\Phi_r(x,f)=\int_\ZT \varphi_r(x-t)f(t)dt, \quad 0<r<1,
\emd
where the kernels $\varphi_r(x)\ge 0$ are arbitrary functions, satisfying
\md1\label{c1}
\int_\ZT\varphi_r(t)=1,\quad 0<r<1,
\emd
and for any numbers $\varepsilon>0$, $0<\tau<1$ there exists  $\delta >0$ such that
\md1\label{0-3}
 \int_e \varphi_r(t)dt<\varepsilon,\quad 0<r<\tau,
\emd
for any measurable $e\subset \ZT$ with $|e|<\delta$. We note, that \e{0-3} is an ordinary absolute continuity condition.  for example,  it is satisfied if $\sup_{0<r<\tau}\|\varphi_r\|_\infty<\infty$ for any $0<\tau<1$. It is clear, that the Poisson kernels satisfy these conditions.
\begin{theorem}\label{T1}
Let $\{\varphi_r\}$ be a family of nonnegative functions, possessing the properties \e{c1} and \e{0-3}. If a function $\lambda(r)\in C[0,1]$ satisfies the conditions
\md1\label{c2}
\lambda(1)= 0,\quad \beta=\limsup_{\delta \to 0}\liminf_{r\to 1}\int_{-\delta \lambda(r)}^{\delta \lambda (r)}\varphi_r(t)dt>\frac{1}{2},
\emd
then there exists a measurable set $E\subset \ZT$ such that
\md0
\limsup_{r\to 1}\Phi_r\left(x+\lambda(r),\ZI_E\right)- \liminf_{r\to 1}\Phi_r\left(x+\lambda(r),\ZI_E\right)\ge2\beta-1.
\emd
\end{theorem}
Observe that if $\varphi_r$ are the Poisson kernels and $\lambda(r)$ satisfies \e{0-1}, then we have $\beta=1$. Thus we get
\begin{corollary}
For any function $\lambda(r):[0,1]\to \ZR$ satisfying \e{0-1}, there exists a harmonic function $u(z)$ on the unit disc with $0\le u(z)\le 1$, such that
 \md0
\limsup_{r\to 1} u\left(re^{i\left(x+\lambda(r)\right)}\right)=1,\quad \liminf_{r\to 1} u\left(re^{i\left(x+\lambda(r)\right)}\right)=0,
 \emd
 at any point $x\in \ZT$.
\end{corollary}
\begin{theorem}\label{T2}
If a family of functions $\{\varphi_r\}$ satisfies \e{Phi}, \e{c1} and for $\lambda(r)\in C[0,1]$ we have $\beta=1$, then there exists a  function $B\in L^\infty(\ZT)$ which is boundary function of a Blaschke product such that the limit
\md0
\lim_{r\to 1}\Phi_r\left(x+\lambda(r),B\right)
\emd
does not exist for any $x\in\ZT$.
\end{theorem}
This theorem generalizes \otrm{OT1}.

\section{Proof of \trm{T1}}
We shall consider the sets
\md1\label{p4}
U(n,\delta)=\bigcup_{j=0}^{n-1}\left(\frac{\pi (2j-\delta)}{n},\frac{\pi (2j+\delta)}{n}\right)\subset \ZT
\emd
in the proofs of this as well as the next theorems.
Using the definition of $\beta$ and  the absolute continuity property \e{0-3}, we may choose numbers  $\delta_k$, $u_k$, $v_k$ $(k\in\ZN)$, satisfying
\md3
& \,\,\,\delta_k<2^{-k-5},\quad 1>v_k>u_k\to 1,\quad 3\lambda(v_k)\le \lambda(u_k)<\pi,\label{p9}\\
&\int_{-\delta_k \lambda(u_k)}^{\delta_k \lambda(u_k)}\varphi_{u_k}(t)dt>\beta(1-2^{-k}), \quad k=1,2,\ldots ,\label{p8}\\
&\int_e|\varphi_r(t)|dt<2^{-k},\label{p10}
\emd
where the last bound holds whenever
\md1\label{p31}
0<r<v_k,\quad |e|\le 10\pi \sum_{j\ge k+1}\sqrt[4]{\delta_j}.
\emd
We shall consider the same sequences \e{p9} with properties \e{p8}-\e{p31} in the proof of \trm{T2} as well. We note that $\sqrt[4]{\delta_j}$ in \e{p31} is necessary only in the proof of \trm{T2}, but for \trm{T1}  just $\delta_j$ is enough.
Denote
\md1\label{c5}
U_k=U(n_k,5\delta_k),\quad n_k=\left[\frac{5\pi}{\lambda(u_k)}\right],\quad k\in \ZN,
\emd
and define the sequence of measurable sets $E_k\subset \ZT$ by
\md3
&E_1=U_1,\\
&E_k=\left\{
\begin{array}{lc}
E_{k-1}\setminus U_k &\hbox{ if }k \hbox{ is even},\\
E_{k-1}\cup U_k&\hbox{ if }k \hbox{ is odd}.
\end{array}\label{c6}
\right.
\emd
It is easy to observe, that if $k<m$, then
\md1\label{c17}
\|\ZI_{E_k}-\ZI_{E_m}\|_1=|E_k\bigtriangleup E_m|\le \sum_{j\ge k+1}|U_j|.
\emd
This implies that $\ZI_{E_n}(t)$ converges to a function $f(t)$ in $L^1$-norm. Using Egorov's theorem, we conclude that $f(t)=\ZI_E(t)$ for some measurable set $E\subset \ZT$. Tending  $m$ to infinity,  from \e{c17} we get
\md1\label{c11}
|E\bigtriangleup E_k|=|(E\setminus E_k)\cup(E_k\setminus E)|\le \sum_{j\ge k+1}|U_j|\le  10\pi \sum_{j\ge k+1}\delta_j.
\emd
Take an arbitrary $x\in \ZT$. There exists an integer  $1\le j_0\le n_k$ such that
\md0
\frac{2\pi j_0}{n_k}-x\in \left[\frac{2\pi }{n_k},\frac{4\pi }{n_k}\right]\subset \left[\frac{\lambda(u_k)}{3},\lambda(u_k)\right]\subset [\lambda(v_k),\lambda(u_k)]
\emd
and therefore, since $\lambda(r)$ is continuous, we may find a number $r$, $u_k\le r\le v_k$, such that
\md1\label{g1}
\lambda(r)=\frac{2\pi j_0}{n_k}-x.
\emd
If  $k\in \ZN$ is odd, then according to the definition of $E_k$ we get
\md0
E_k\supset U_k\supset I=\left(\frac{\pi(2 j_0+5\delta_k)}{n_k},\frac{\pi(2 j_0-5\delta_k)}{n_k}\right).
\emd
Thus, using \e{p8}, \e{g1} as well as the definition of $n_k$ from \e{c5}, we conclude
\md9\label{g15}
\Phi_r(x+\lambda(r),\ZI_{E_k})&\ge \int_I\varphi_r(x+\lambda(r)-t)dt\\
&= \int_I\varphi_{r}\left(\frac{2\pi j_0}{n_k}-t\right)dt\\
&=\int_{-5\pi \delta_k/n_k}^{5\pi \delta_k/n_k}\varphi_{r}\left(t\right)dt\\
&\ge \int_{-\delta_k \lambda(u_k)}^{\delta_k \lambda(u_k)}\varphi_{r}(t)dt>\beta(1-2^{-k}).
\emd
From \e{p10} and \e{c11} it follows that
\md0
\left|\Phi_r\left(t,\ZI_E\right)-\Phi_{r}\left(t,\ZI_{E_k}\right)\right|<2^{-k},\quad t\in\ZT,\quad 0<r<v_k,
\emd
and hence from \e{g15} we obtain
\md1\label{c8}
\limsup_{r\to 1}\Phi_r\left(x+\lambda(r),\ZI_E\right)\ge \beta.
\emd
If $k\in\ZN$ is even, then we have $E_k\cap U_k=\varnothing$ and therefore $E_k\cap I=\varnothing$.
Thus we get
\md2
\Phi_r(x+\lambda(r),\ZI_{E_k})&\le  \int_\ZT\varphi_r(x+\lambda(r)-t)dt-\int_I\varphi_r(x+\lambda(r)-t)dt\\
&\le 1- \int_{-\delta_k \lambda(u_k)}^{\delta_k \lambda(u_k)}\varphi_r\left(t\right)dt\le 1-\beta(1-2^{-k})
\emd
and similarly we get
\md1\label{c9}
\liminf_{r\to 1}\Phi_r\left(x+\lambda(r),\ZI_E\right)\le1-\beta.
\emd
Relations \e{c8} and \e{c9} complete the proof of theorem.

\section{Proof of \trm{T2}}
The following finite Blaschke products
\md1\label{p1}
b(n,\delta,z)=\frac{z^n-\rho^n}{\rho^nz^n-1}=\prod_{k=0}^{n-1}\frac{z-\rho  e^{\frac{2\pi ik}{n}}}{\rho  e^{\frac{2\pi ik}{n}}z-1},\quad \rho=e^{-\sqrt{\delta}/n}.
\emd
plays significant role in the proof of \trm{T2}. Similar products are used in the proof of theorem \otrm{OT1} too. If $z=e^{ix}$, then \e{p1} defines a continuous function in $H^\infty(\ZT)$. We shall use the set $U(n,\delta)$ defined in \e{p5}. The following lemma shows that on $U(n,\delta)$ the function \e{p1} is approximative $-1$, and outside of $U(n,\sqrt[4]{\delta})$ is approximative $1$.
\begin{lemma}\label{L1}
There exists an absolute constant $C>0$ such that
\md3
&\left|b\left(n,\delta,e^{ix}\right)+1\right|\le C\sqrt{\delta},\quad x\in U(n,\delta),\label{p5}\\
&\left|b\left(n,\delta,e^{ix}\right)-1\right|\le C\sqrt[4]{\delta},\quad x\in \ZT\setminus U(n,\sqrt[4]{\delta}).\label{p23}
\emd
\end{lemma}
\begin{proof}
Deduction of these inequalities based on the inequalities
\md0
\frac{|z|}{2}\le |e^z-1|\le 2|z|, \quad z\in\mathbb{C}.
\emd
If  $x\in U(n,\delta)$, then we have
\md9\label{p5}
\left|b\left(n,\delta,e^{ix}\right)+1\right|&=\left|\frac{(e^{inx}-1)(\rho^n+1)}{\rho^ne^{inx}-1}\right|\le \frac{4\pi \delta}{1-e^{-\sqrt{\delta}}},\\
&\le \frac{4e\pi \delta}{e^{\sqrt{\delta}}-1}\le\frac{8e\pi \delta}{\sqrt{\delta}}\le C\sqrt{\delta}.
\emd
If $x\in \ZT\setminus U(n,\sqrt[4]{\delta})$, then $e^{inx}=e^{i\alpha}$ with $\pi\sqrt[4]{\delta}<|\alpha|<\pi $. Thus we obtain
\md9\label{p23}
\left|b\left(n,\delta,e^{ix}\right)-1\right|&= \left|\frac{(e^{inx}+1)(1-\rho^n)}{\rho^ne^{inx}-1}\right|= \frac{2(e^{\sqrt{\delta}}-1)}{|e^{inx}-e^{\sqrt{\delta}}|}\\
&\le \frac{4 \sqrt{\delta}}{|e^{inx}-1|-|e^{\sqrt{\delta}}-1|}\le\frac{4 \sqrt{\delta}}{\pi\sqrt[4]{\delta}/2-2\sqrt{\delta}}\le C\sqrt[4]{\delta}.
\emd
\end{proof}
\begin{proof}[Proof of \trm{T2}]
First we choose numbers  $\delta_k$, $u_k$, $v_k$ $(k\in\ZN)$, satisfying \e{p9}-\e{p10} with $\beta=1$. Then we
denote
\md1\label{p22}
b_k(x)=b(n_k,\delta_k,e^{ix}), \quad n_k=\left[\frac{6\pi}{\lambda(u_k)}\right],\quad k\in \ZN,
\emd
and
\md0
B_k(x)=\prod_{j=1}^kb_j(x),\quad B(x)=\prod_{j=1}^\infty b_j(x).
\emd
The convergence of the infinite product follows from the bound \e{p3}, which will be obtained bellow. Observe that in the process of selection of the numbers \e{p9} we were free to define $\delta_k>0$ as small as needed. Besides, taking $u_k$ to be close to $1$ we may get $n_k$ as big as needed. Using these notations and \lem{L1}, aside of the conditions \e{p9}-\e{p10} we can additionally claim the bounds
\md3
&\omega\left(2\pi/n_k,B_{k-1}\right)=\sup_{|x-x'|<2\pi/n_k}|B_{k-1}(x)-B_{k-1}(x')|<2^{-k},\label{p11}\\
&\left|b_k(x)+1\right|< 2^{-k}, \quad x\in U(n_k,6\delta_k),\label{p2}\\
&\left|b_k(x)-1\right|<2^{-k},x\in\ZT\setminus U(n_k,\sqrt[4]{\delta_k}). \label{p3}
\emd
From \e{p3} we get
\md9\label{p7}
\left|B(x)-B_k(x)\right|&=\left|\prod_{j\ge k+1}b_j(x)-1\right|\\
&\le \prod_{j\ge k+1}(1+2^{-j})-1<2^{-k+1},\quad x\in \ZT\setminus \bigcup_{j\ge k+1}U\left(n_j,\sqrt[4]{\delta_j}\right).
\emd
 Take an arbitrary $x\in \ZT$. There exists an integer  $1\le j_0\le n_k$ such that
\md0
\frac{2\pi j_0}{n_k}-x\in \left[\frac{2\pi }{n_k},\frac{4\pi }{n_k}\right]\subset  \left[\frac{2\pi }{n_k},\frac{5\pi }{n_k}\right]\subset\left[\frac{\lambda(u_k)}{3},\lambda(u_k)\right]\subset [\lambda(v_k),\lambda(u_k)],
\emd
where the inclusions follow from the definition of $n_k$ (see \e{p22}) and from the inequality $3\lambda(v_k)\le \lambda(u_k)<\pi$ coming from \e{p9}. Thus since $\lambda(r)$ is continuous, we may find numbers $u_k\le r'\le r''\le v_k$, such that
\md1\label{p30}
\lambda(r')=\frac{2\pi j_0}{n_k}-x,\quad \lambda(r'')=\frac{2\pi j_0}{n_k}+\frac{\pi }{n_k}-x.
\emd
For the set
\md0
e=\bigcup_{j\ge k+1}U\left(n_j,\sqrt[4]{\delta_j}\right),
\emd
we have
\md0
|e|=10\pi \sum_{j\ge k+1}\sqrt[4]{\delta_j}.
\emd
So taking $r\in [u_k,v_k]$, from \e{p10} and  \e{p7} we conclude
\md9\label{p20}
\big|\Phi_r(x,B)&-\Phi_r(x,B_k)\big|\\
&\le \int_e\varphi_r(x-t)|B(t)-B_k(t)|dt+2^{-k+1} \int_{ \ZT\setminus e}\varphi_r(x-t)dt\\
&\le 2\cdot 2^{-k}+2^{-k+1}=4\cdot 2^{-k} ,\quad x\in \ZT.
\emd
If
\md0
t\in I=(-\delta_k\lambda(u_k),\delta_k\lambda(u_k))\subset \left(-\frac{6\pi\delta_k}{n_k},\frac{6\pi\delta_k}{n_k}\right),
\emd
then we have
\md2
&\frac{2\pi j_0}{n_k}-t\in U(n_k,6\delta_k),\\
 &\frac{2\pi j_0}{n_k}+\frac{\pi}{n_k}-t\in\ZT\setminus U(n_k,\sqrt[4]{\delta_k}).
\emd
Then, using these relations, \e{p2} and \e{p11}, we get
\md9\label{p12}
\left|B_k\left(\frac{2\pi j_0}{n_k}-t\right)\right.&+\left. B_{k-1}\left(\frac{2\pi j_0}{n_k}\right)\right|\\
&\le\left|B_{k-1}\left(\frac{2\pi j_0}{n_k}-t\right)\right|\left|b_k\left(\frac{2\pi j_0}{n_k}-t\right)+1\right|\\
&+\left|B_{k-1}\left(\frac{2\pi j_0}{n_k}-t\right)-B_{k-1}\left(\frac{2\pi j_0}{n_k}\right)\right|\\
&< 2^{-k}+2^{-k}=2^{-k+1}
\emd
and
\md9\label{p13}
\left|B_k\left(\frac{2\pi j_0}{n_k}+\frac{\pi}{n_k}-t\right)\right.&-\left. B_{k-1}\left(\frac{2\pi j_0}{n_k}\right)\right|\\
&\le\left|B_{k-1}\left(\frac{2\pi j_0}{n_k}+\frac{\pi}{n_k}-t\right)\right|\left|b_k\left(\frac{2\pi j_0}{n_k}+\frac{\pi}{n_k}-t\right)-1\right|\\
&+\left|B_{k-1}\left(\frac{2\pi j_0}{n_k}+\frac{\pi}{n_k}-t\right)-B_{k-1}\left(\frac{2\pi j_0}{n_k}\right)\right|\\
&< 2^{-k}+2^{-k}=2^{-k+1}.
\emd
On the other hand, using \e{p8}, \e{p30} and \e{p12}, we get
\md9\label{c15}
\left|\Phi_{r'}(x+\lambda(r'),B_k)\right.&+\left.B_{k-1}\left(\frac{2\pi j_0}{n_k}\right)\right|\\
= & \left|\int_\ZT\varphi_{r'}(t)B_k(x+\lambda(r')-t)dt+B_{k-1}\left(\frac{2\pi j_0}{n_k}\right)\right|\\
=&  \left|\int_\ZT\varphi_{r'}(t)\left[B_k\left(\frac{2\pi j_0}{n_k}-t\right)+B_{k-1}\left(\frac{2\pi j_0}{n_k}\right)\right]dt\right|\\
\le&  \left|\int_I\varphi_{r'}(t)\left[B_k\left(\frac{2\pi j_0}{n_k}-t\right)+B_{k-1}\left(\frac{2\pi j_0}{n_k}\right)\right]dt\right|\\
+&\left|\int_{I^c}\varphi_{r'}(t)\left[B_k\left(\frac{2\pi j_0}{n_k}-t\right)+B_{k-1}\left(\frac{2\pi j_0}{n_k}\right)\right]dt\right|\\
\le& 2^{-k+1}\int_I\varphi_{r'}(t)dt+2\cdot 2^{-k}\le 4\cdot 2^{-k}.
\emd
Similarly, using \e{p13}, we conclude
\md1\label{c16}
\left|\Phi_{r''}\left(x+\lambda(r''),B_k\right)-B_{k-1}\left(\frac{2\pi j_0}{n_k}\right)\right|\le 4\cdot 2^{-k}.
\emd
From \e{p20}, \e{c15} and \e{c16} it follows that
\md0
\left|\Phi_{r'}(x+\lambda(r'),B)-\Phi_{r''}\left(x+\lambda(r''),B\right)\right|\ge 1-16\cdot 2^{-k},
\emd
which implies the divergence of $ \Phi_{r}(x+\lambda(r),B)$ at a point $x$. The theorem is proved.
\end{proof}

 \end{document}